\documentclass[reqno]{amsart}
\usepackage{amssymb}
\usepackage[usenames,dvipsnames]{color}
\usepackage{hyperref}
\hypersetup{hidelinks}
\usepackage{tikz}
\usetikzlibrary{arrows.meta,calc,patterns}

\usepackage{todonotes}

\numberwithin{equation}{section}

\newtheorem{theorem}{Theorem}[section]
\newtheorem{corollary}[theorem]{Corollary}
\newtheorem{lemma}[theorem]{Lemma}

\theoremstyle{definition}
\newtheorem{remark}[theorem]{Remark}

\def\bR{\mathbb{R}}
\def\cC{\mathcal{C}}
\def\cH{\mathcal{H}}
\def\cQ{\mathcal{Q}}
\def\cE{\mathcal{E}}
\def\cG{\mathcal{G}}

\def\epsi{\varepsilon}

\begin{document}

\title[One-dimensional parabolic equations]{On one-space dimensional parabolic equations with measurable coefficients: Sobolev estimates and the Alexandrov maximum principle}

\author[H. Dong]{Hongjie Dong}
\address[H. Dong]{Division of Applied Mathematics, Brown University, 182 George Street, Providence, RI 02912, USA}
\email{Hongjie\_Dong@brown.edu}

\thanks{H. Dong was partially supported by the NSF under agreements DMS-2350129.}

\author[Z. Li]{Zongyuan Li}
\address[Z. Li]{Department of Mathematics, City University of Hong Kong, 83 Tat Chee Avenue, Kowloon Tong, Hong Kong SAR, China}
\email{zongyuan.li@cityu.edu.hk}

\thanks{Z. Li was partially supported by the Hong Kong RGC grant ECS 21307225.}


\begin{abstract}
Let $0<\kappa<1$ and set $p_+=2/(1-\kappa)$ and $p_-=2/(1+\kappa)$.  We construct a coefficient $\kappa\leq a\leq\kappa^{-1}$, smooth outside a compact set of Lebesgue measure zero, for which the $W^{1,2}_{p_+}$ estimate fails for the one-space dimensional nondivergence form parabolic equation. The corresponding solution has second spatial derivative in the weak $L_{p_+}$ space, but not in $L_{p_+}$. By duality, the $W^{1,2}_{p_-}$ a priori estimate also fails.  Conversely, we demonstrate that
the $W^{1,2}_p$ estimate holds and the equation is uniquely solvable for $|p-2|<c\kappa$, showing that the size of the solvability interval around $2$ has optimal order $\kappa$ as $\kappa\downarrow0$, which addresses a question raised in \cite{Krylov2016}. 
These results imply that the Alexandrov maximum principle holds for $p>2-c\kappa$ and this order is sharp as $\kappa\to 0$. The corresponding results for divergence form equations are also obtained.
\end{abstract}

\maketitle

\section{Introduction}

We consider one-dimensional parabolic equations with a bounded measurable coefficient $a=a(t,x)$ satisfying
\begin{equation}
 \kappa\leq a(t,x)\leq\kappa^{-1},\qquad 0<\kappa<1.                 \label{eq:intro-ell}
\end{equation}
The two equations considered in this paper are the nondivergence equation
\begin{equation}
 u_t-a u_{xx}=f                                           \label{eq:intro-nondiv}
\end{equation}
and the divergence equation
\begin{equation}
 v_t-(a v_x)_x=f_x.                                      \label{eq:intro-div}
\end{equation}
We are interested in the Sobolev type estimates of the form
\begin{equation}
 \|u_t\|_{L_p}+\|u_{xx}\|_{L_p}\leq N\|f\|_{L_p}       \label{eq:intro-est-nondiv}
\end{equation}
and
\begin{equation}
 \|v_x\|_{L_p}\leq N\|f\|_{L_p},                       \label{eq:intro-est-div}
\end{equation}
respectively. Here and below lower-order terms and harmless $L_p$ norms of the solution are omitted when they are not relevant to the discussion.

The exponent $p=2$ is special for both equations. For \eqref{eq:intro-nondiv}, integration by parts gives the $W^{1,2}_2$ estimate without any regularity assumption on $a$, while for \eqref{eq:intro-div} one has the usual energy estimate.  If the coefficients have additional regularity, for instance VMO or partially VMO regularity, estimates for all $p\in(1,\infty)$ are well known; see, among many references, \cite{Krylov2007,Krylov2008,Dong2020}. In particular, the coefficient $a$ is allowed to be merely measurable with respect to $t$ or $x$ and has small mean oscillation with respect to the other variable in small cylinders. For coefficients measurable in both variables, the situation is different. Krylov \cite{Krylov2016} proved the $W^{1,2}_p$ estimate and solvability for \eqref{eq:intro-nondiv} when $p$ is in a small neighborhood of $2$ which depends on the ellipticity constant, and gave counterexamples showing nonuniqueness for $1<p<3/2$ and nonsolvability for $p>3$. The construction in \cite{Krylov2016} is based on estimates from below and above for solutions of the Cauchy problem with initial data given by an indicator function of a small interval, and then a delicate construction of explicit solutions to the Ornstein--Uhlenbeck equations. The validity of the \(W^{1,2}_p\) estimates in the entire range \(p\in[3/2,3]\) was left open.
For divergence equations, parabolic Meyers-type estimates give higher integrability of the spatial gradient slightly above $2$, and by duality also slightly below $2$. See \cite{GiaquintaStruwe1982,ABESParabolic2019}. We also mention the recent paper \cite{BMV24}, in which the authors constructed counterexamples to maximal regularity for divergence form parabolic equations in space dimensions two and higher.

In this paper, we first give a quantitative positive result. There is an absolute constant $c>0$ such that, for every coefficient satisfying \eqref{eq:intro-ell}, the zero initial value problem for \eqref{eq:intro-nondiv} and \eqref{eq:intro-div} on $\cQ=(0,1)\times\bR$ is uniquely solvable in $W^{1,2}_p$ and $\cH^1_p(\cQ)$, respectively, whenever
\begin{equation}
 |p-2|<\delta_\kappa,\qquad \delta_\kappa=c\kappa.             \label{eq:intro-positive}
\end{equation}
This follows from the quantitative Shneiberg theorem of Auscher, Bortz, Egert, and Saari \cite[Theorem A.1]{ABESShneiberg}, applied with coefficient-adapted norms.

Our main negative result gives an explicit obstruction at the exponents
\begin{equation*}
 p_+=\frac{2}{1-\kappa},\qquad p_-=p_+'=\frac{2}{1+\kappa}.    
\end{equation*}
For $p_+$ we construct a coefficient satisfying \eqref{eq:intro-ell} and a $W^{1,2}_2$ solution of the homogeneous equation \eqref{eq:intro-nondiv} in $Q_1$ such that $u,u_x$ are bounded and continuous in $Q_1$, but
\begin{equation*}
 u_t,u_{xx}\in L_{p_+,\infty}(D)\setminus L_{p_+}(D)
\end{equation*}
for a nonempty open rectangle $D\subset\subset Q_{1/2}$
For the same fixed coefficient we also construct solutions $\phi_j\in C_0^\infty(Q_1)$ for which the right-hand side in \eqref{eq:intro-nondiv} as well as $\phi_j$ and $\phi_{j,x}$ stays uniformly bounded in $L_{p_+}(Q_1)$ while
\[
 \|\phi_{j,xx}\|_{L_{p_+}(D)}\to\infty.
\]
Thus the estimate \eqref{eq:intro-est-nondiv} fails at $p_+$.  Moreover, the coefficient can be chosen smooth outside a compact set of Lebesgue measure zero.  By a one-dimensional duality argument, 
the corresponding estimate fails at $p_-$.  Hence, if the ellipticity parameter is allowed to depend on $p$, the a priori estimate fails for every $p\neq2$.  It is worth noting that for $p<3/2$, Krylov's result in \cite{Krylov2016} is stronger as it gives nonuniqueness for the homogeneous problem of the Cauchy problem.

The positive and negative results have the same order in the ellipticity parameter.  Indeed,
\[
 p_+-2=\frac{2\kappa}{1-\kappa},\qquad 2-p_-=\frac{2\kappa}{1+\kappa}.
\]
Together with \eqref{eq:intro-positive}, this shows that the largest coefficient-uniform neighborhood of $2$ has size of order $\kappa$ as $\kappa\downarrow0$.

The divergence equation follows from the same construction.  If $v=u_x$, then
\[
 v_t-(a v_x)_x=0
\]
in the weak sense and $v_x=u_{xx}$.  Thus the estimate \eqref{eq:intro-est-div} fails at $p_+$, and duality gives the corresponding failure at $p_-$. 

For nondivergence form elliptic equations, counterexamples of the $W^2_p,p\neq 2,$ estimate with piecewise constant coefficients in each quadrant in $\bR^2$ were obtained by Dong and Kim \cite{DongKim2014}. For divergence form equations, see an earlier construction in Piccinini and Spagnolo \cite{PS72}. Another example due to Ural'tseva \cite{MR0226179} shows the impossibility of the $W^2_p$ estimate when the dimension $d\ge 2$ and $p\neq 2$ even if the coefficients are continuous except at a single point ($d=2$) or a line ($d=3$). See also \cite{MR2667637}. For divergence form equation in $\bR^2$, see also the seminal work by Serrin \cite{Serrin64}. In \cite{MR1612401}, Nadirashvili showed that the weak uniqueness for martingale problems may fail if coefficients are merely measurable and $d \ge 3$.

Our construction is in the spirit of Astala, Faraco, and Sz\'ekelyhidi \cite{AFS2008}, where the convex integration method, originally due to M\"uller and \v Sver\'ak \cite{MS03}, was used to find counterexamples for both divergence and nondivergence form elliptic equations in two dimensions, with the same critical exponents $p_+$ and $p_-$ as in this paper. See also the references therein for earlier results.
In the present parabolic setting, the two convex splittings are realized separately by spatial and temporal perturbations. 
A notable feature of the construction is that the spatial and temporal scales of the sets $E_n$ decrease at comparable rates, rather than according to the natural parabolic relation in which the temporal scale is the square of the spatial one. In terms of the splitting parameters, the parabolic relation would correspond to $\eta_n\simeq 2\theta_n$, whereas in our construction $\eta_n\simeq\theta_n$. This gives the iteration an elliptic feature. Interestingly, imposing the parabolic relation instead leads, at the level of a formal scaling calculation, to the critical exponent $3/(1-\kappa)$, which approaches $3$ as $\kappa\to0$, the threshold appearing in Krylov's example. Recently, the convex integration method has also been used to study Sobolev estimates for the $p$-Laplace equation. See \cite{CT25,Schikorra26}.

As an application of our main results, we obtain a quantitative $p$ range for the Alexandrov maximum principle for one-space dimensional parabolic equations with measurable coefficients. 
In history, for uniformly parabolic equations in $d$ spatial dimensions, the Alexandrov maximum principle was established by Krylov for $p = d+1$ in \cite{KrylovParabolic1976}. Later, this was improved to $p = d+1-\epsi$ by Escauriaza \cite{Escauriaza1993}, following the work of Fabes and Stroock \cite{FabesStroock1984} on elliptic equations. For other related works, see Crandall et al. \cite{CrandallFokKocanSwiech1998} and Cabr\'e \cite{CabreABP1995}.

A natural question is how far the exponent can be lowered.
In view of the parabolic Sobolev embedding $W^{1,2}_p \hookrightarrow C^0$, 
we need $p > (d+2)/2$ even for the heat equation.
The optimal lower endpoint of $p$, however, appears to be unknown even in one space dimension. 
In \cite{Krylov2016}, Krylov constructed examples showing failure for $p$ below the critical Sobolev exponent $3/2$, if the ellipticity constant is allowed to depend on $p$, leaving open the optimal lower endpoint of $p$. See the discussions in \cite[Remark~2.2]{Krylov2016}. Our results show that for a fixed ellipticity constant $\kappa$, the Alexandrov maximum principle holds for $p > 2-c\kappa$, and this order is sharp as $\kappa \rightarrow 0$. In particular, the critical exponent is ellipticity-dependent and approaches $2$ as the ellipticity degenerates.

The remainder of the paper is organized as follows.  In Section \ref{sec:main} we state the main results.  The positive result near $2$ is proved in Section \ref{sec:positive}.  Section \ref{sec:construction} contains the key construction of the counterexample. The proofs of the negative results on Sobolev type estimates are given in Sections \ref{sec:estfailure} and \ref{sec:duality}. In Section \ref{sec-alex}, we prove both positive and negative results for the Alexandrov maximum principle.

\section{Main results}                                      \label{sec:main}

For the positive and solvability results, we consider the domain
\[
 \cQ=(0,1)\times\bR.
\]
Set
\[
 \mathring W^{1,2}_p(\cQ)
 =\{u\in W^{1,2}_p(\cQ):u(0,\cdot)=0\},
\]
where
\[
 W^{1,2}_p(\cQ)=\{u:u,u_x,u_t,u_{xx}\in L_p(\cQ)\}.
\]
For divergence form equations, we consider
\begin{equation*}
    \cH^1_p(\cQ) = \{u: u,  u_x, (1-\partial_{xx})^{-1/2}u_t \in L_p(\cQ)\}.
\end{equation*}
The weak $L_p$ space is denoted by $L_{p,\infty}$. We also work with the rectangle
\[
 Q=Q_1=(-1,0)\times(-1,1),
\]
with its parabolic boundary denoted by $\partial_p Q_1 := (\{-1\}\times (-1,1)) \cup ((-1,0) \times \{-1,1\})$.

\begin{theorem}            \label{thm:positive}
There is an absolute constant $c>0$ with the following property.  Let $0<\kappa<1$ and let $a$ be any measurable function on $\cQ$ satisfying \eqref{eq:intro-ell}.  If
\begin{equation}
 |p-2|<c\kappa,                                      \label{eq:positive-range}
\end{equation}
then for every $f\in L_p(\cQ)$ there is a unique $u\in\mathring W^{1,2}_p(\cQ)$ satisfying
\begin{equation*}
 u_t-a u_{xx}=f\quad\text{in }\cQ,\qquad u(0,\cdot)=0, 
\end{equation*}
and
\begin{equation}
 \|u_t\|_{L_p(\cQ)}+\|u_{xx}\|_{L_p(\cQ)}
 \leq N_{\kappa,p}\|f\|_{L_p(\cQ)}.                       \label{eq:positive-est}
\end{equation}

Similarly, there exists a unique weak solution $v \in \cH^1_p(\cQ)$ of
\begin{equation}\label{eqn-260916-1128}
    v_t - (av_x)_x = f_x
    \quad
    \text{in}\,\,\cQ,
    \qquad
    v(0,\cdot) = 0,
\end{equation}
and
\begin{equation*}
    \|v\|_{\cH^1_p(\cQ)} \leq N_{\kappa, p} \|f\|_{L_p(\cQ)}.
\end{equation*}

The same results also hold for initial-boundary value problems on $Q_1$ with zero data on $\partial_p Q_1$.
\end{theorem}

Here, $v \in \cH^1_p(\cQ)$ is weak solution of \eqref{eqn-260916-1128}, if
\begin{equation*}
    \int_{\cQ} v \varphi_t - a v_x \varphi_x \, dx dt = \int_{\cQ} f\varphi_x\,dxdt,
    \quad
    \forall \varphi \in C^\infty_c(\overline \cQ), \,\, \varphi(1,\cdot) = 0.
\end{equation*}

The critical exponents in the counterexample below satisfy
\[
 p_+-2=\frac{2\kappa}{1-\kappa},\qquad 2-p_-=\frac{2\kappa}{1+\kappa}.
\]
Thus Theorem \ref{thm:positive} has the optimal order $\kappa$ as $\kappa\downarrow0$, although we do not identify the optimal endpoints of the solvability interval.

For the counterexample, we fix a nonempty open rectangle $D\subset\subset Q_{1/2} =(-1/4,0)\times(-1/2,1/2)$.
\begin{theorem}  \label{thm:main}
Let $0<\kappa<1$ and $p=p_+=2/(1-\kappa)$.  There exists a Borel function $a:Q\to[\kappa,\kappa^{-1}]$, such that the following hold.

(i) There exists a solution $u \in W^{1,2}_2(Q)$ of
\begin{equation*}
u_t - a u_{xx} = 0,\quad \text{a.e. in}\,\,Q,
\end{equation*}
such that $u,u_x$ are bounded and continuous in $Q$ and 
$$
c \tau^{-p} \leq | \{(t,x) \in D: u_{xx}(t,x) > \tau\}| \leq N \tau^{-p}
$$ 
for all sufficiently large $\tau>0$.
In particular, we have
\begin{equation*}
u_t, u_{xx} \in L_{p,\infty}(D),\quad u_t, u_{xx} \notin L_p(D).
\end{equation*}

(ii) There exist $\{\phi_j\}_j \subset C^\infty_0(Q)$, such that
\begin{equation}
 \|\phi_j\|_{L_\infty}+\|\phi_{j,x}\|_{L_\infty}
 +\|\phi_{j,t}-a\phi_{j,xx}\|_{L_p(Q)}\leq N,         \label{eq:test-upper}
\end{equation}
whereas
\begin{equation}
 \int_{D}|\phi_{j,xx}|^p\,dx\,dt\geq c\log j.        \label{eq:test-lower}
\end{equation}
\end{theorem}
In fact, the coefficient $a$ and the solution $u$ constructed in (i) are smooth outside a compact set $\Sigma \subset D$ of measure zero. 

By considering $v = u_x$ and $v_j = \phi_{j,x}$, we immediately obtain examples for divergence form equations.
\begin{corollary}                 \label{cor:div}
For the coefficient in Theorem \ref{thm:main} and same exponent $p$, there exists a bounded continuous weak solution $v$ of
\begin{equation*}
 v_t-(a v_x)_x=0\quad\text{in }Q,                       
\end{equation*}
with
\begin{equation*}
 v_x\in L_2(Q)\cap L_{p,\infty}(Q),\qquad
 v_x\notin L_p(D).                                 
\end{equation*}
Furthermore, there exist a sequence of smooth solutions $v_j\in C_0^\infty(Q)$ satisfying
\begin{equation*}
 v_{j,t}-(a v_{j,x})_x=(f_j)_x,
 \qquad
 \|f_j\|_{L_p(Q)}+\|v_j\|_{L_p(Q)}\leq N,                 
\end{equation*}
while for all sufficiently large $j$,
\begin{equation*}
 \int_{D}|v_{j,x}|^p\,dx\,dt\geq c\log j.      
\end{equation*}
\end{corollary}

We also show the failure of the estimates at the lower exponent $p_-$  for both non-divergence and divergence form equations.
\begin{theorem}                     \label{cor:dual}
Let $p_-=2/(1+\kappa)$.  
There is a Borel function $\widetilde a$ on $\cQ=(0,1)\times\bR$, satisfying $\kappa\leq\widetilde a\leq\kappa^{-1}$, 
for which the $W^{1,2}_{p_-}$ a priori estimate fails.  More precisely, there is no constant $N$ such that, for every $f\in L_2(\cQ)\cap L_{p_-}(\cQ)$, the $W^{1,2}_2$ solution of
\[
 u_t-\widetilde a u_{xx}=f\quad\text{in }\cQ,
 \qquad u(0,\cdot)=0,
\]
satisfies
\begin{equation}
 \|u_t\|_{L_{p_-}(\cQ)}+\|u_{xx}\|_{L_{p_-}(\cQ)}
 \leq N\|f\|_{L_{p_-}(\cQ)}.                         \label{eq:lower-forward}
\end{equation}

For the same coefficient, there is no constant $N$ such that, for every $f\in L_2(\cQ)\cap L_{p_-}(\cQ)$, the weak solution of
\[
 v_t-(\widetilde a v_x)_x=f_x,\qquad v(0,\cdot)=0,
\]
satisfies
\[
 \|v_x\|_{L_{p_-}(\cQ)}\leq N\|f\|_{L_{p_-}(\cQ)}.
\]
\end{theorem}

\begin{remark}
For any prescribed $p>2$, the choice $\kappa=(p-2)/p$ gives $p=p_+$.  For any prescribed $1<p<2$, the choice $\kappa=(2-p)/p$ gives $p=p_-$.  Thus, if the ellipticity constant is not fixed in advance, the general measurable-coefficient theory has a counterexample for every $p\neq2$.
\end{remark}

As an application, we also obtain a quantitative range $p > 2-c\kappa$ for the Alexandroff maximum principle. Such range has the optimal order as $\kappa \downarrow 0$.
\begin{theorem} \label{thm-alex}
 (i)   There is an absolute constant $c>0$ with the following property.  Let $0<\kappa<1$ and let $a$ be any measurable function on $Q_1$ satisfying \eqref{eq:intro-ell}. For $u \in W^{1,2}_p(Q_1)$ with $p > 2-c\kappa$, we have
    \begin{equation} \label{eqn-260918-1113}
        \sup_{\overline Q_1} u_{+} \leq \sup_{\partial_p Q_1} u_{+} + N\| (u_t-au_{xx})_+\|_{L_p(Q_1)}.
    \end{equation}

(ii)    There is a Borel function $\widetilde a$ on $Q_1$, satisfying $\kappa \leq \widetilde a \leq \kappa^{-1}$, for which the following holds: 
    there is no constant $N$ such that, for every $f \in L_2 (Q_1)$, the $W^{1,2}_2$ solution $v$ of
    \[
        v_t-\widetilde a v_{xx}=f\quad\text{in }Q_1,
    \qquad v =0\quad\text{on }\partial_p Q_1,
    \]
    satisfies
    \begin{equation}\label{eqn-260918-1142}
        \|v\|_{L_1(Q_1)} \leq N \|f\|_{L_{p_-}(Q_1)}.
    \end{equation}
    In particular, the Alexandrov maximum principle fails for every $p \leq p_-$.
\end{theorem}

\section{The positive estimates near \texorpdfstring{$2$}{2}}                      \label{sec:positive}

\begin{proof}[Proof of Theorem~\ref{thm:positive}]
For $1<p<\infty$, equip $\mathring W^{1,2}_p(\cQ)$ with the equivalent norm
\begin{equation*}
 \|u\|_{X_p^a}:=
 \left\|\left(\frac{u_t}{\sqrt a},\sqrt a\,u_{xx}\right)\right\|_{L_p(\cQ;\ell_2^2)}
= \left\| \left( \left|\frac{u_t}{\sqrt a}\right|^2  + \left|\sqrt a\,u_{xx}\right|^2\right)^{1/2}\right\|_{L_p(\cQ)},
\end{equation*}
and equip $L_p(\cQ)$ with
\begin{equation*}
 \|f\|_{Y_p^a}:=\left\|\frac{f}{\sqrt a}\right\|_{L_p(\cQ)}. 
\end{equation*}
Let $T_a u=u_t-a u_{xx}$.  Since
\[
 \frac{T_a u}{\sqrt a}=\frac{u_t}{\sqrt a}-\sqrt a\,u_{xx},
\]
we have
\begin{equation}
 \|T_a u\|_{Y_p^a}\leq \sqrt2\,\|u\|_{X_p^a}             \label{eq:T-bound}
\end{equation}
for every $p$.

At $p=2$, using a standard approximation argument, integration by parts on $(0,1)\times\bR$ and the zero initial condition give
\begin{align*}
 \left\|\frac{T_a u}{\sqrt a}\right\|_2^2
 &=\left\|\frac{u_t}{\sqrt a}\right\|_2^2
   +\|\sqrt a\,u_{xx}\|_2^2-2\int_{\cQ} u_tu_{xx} \notag\\
 &=\left\|\frac{u_t}{\sqrt a}\right\|_2^2
   +\|\sqrt a\,u_{xx}\|_2^2+\|u_x(1,\cdot)\|_2^2.     
\end{align*}
Hence
\begin{equation}
\|u\|_{X_2^a}\leq \|T_a u\|_{Y_2^a}.                   \label{eq:L2-lower}
\end{equation}
By the unique solvability of the heat equation and the method of continuity,
\[
 T_a:\mathring W^{1,2}_2(\cQ)\to L_2(\cQ)
\]
is invertible at $p=2$. 

Next we put this estimate into interpolation spaces and keep track of the interpolation constants. Note that 
\[
\partial_t-\partial_{xx}:\mathring W^{1,2}_p(\cQ)\to L_p(\cQ).
\]
is an isomorphism for any $p\in (1,\infty)$.
Let $X_p^0$ denote $\mathring W^{1,2}_p(\cQ)$ with the unweighted norm
\[
 \|u\|_{X_p^0}=\|(u_t,u_{xx})\|_{L_p(\cQ;\ell_2^2)}.
\]
Then,
\begin{equation}
 \kappa^{1/2}\|u\|_{X_p^0}\leq \|u\|_{X_p^a}
 \leq \kappa^{-1/2}\|u\|_{X_p^0}.                            \label{eq:Xcompare}
\end{equation}
Take $p_0=3/2$ and $p_1=3$, and let
\[
 \widetilde X_\theta=[X_{p_0}^a,X_{p_1}^a]_\theta,
 \qquad
 \widetilde Y_\theta=[Y_{p_0}^a,Y_{p_1}^a]_\theta,
 \qquad
 \frac1p=\frac{1-\theta}{p_0}+\frac\theta{p_1}.
\]
Multiplication by $a^{-1/2}$ identifies $Y_p^a$ isometrically with $L_p(\cQ)$, with the same multiplier for every $p$.  Hence
\begin{equation}
 \widetilde Y_\theta=Y_p^a
 \quad\text{isometrically}.                                      \label{eq:Yinterpolate}
\end{equation}
In particular, no loss of $\kappa$ factor occurs on the $Y$ side.  On the $X$ side,  noting $X_p^0 = [X_{p_0}^0, X_{p_1}^0]_\theta$, interpolation of \eqref{eq:Xcompare} at the two endpoints gives
\begin{equation}
    \label{eq3.13}
 \kappa^{1/2}\|u\|_{X_p^0}
 \leq N\|u\|_{\widetilde X_\theta}
 \leq N\kappa^{-1/2}\|u\|_{X_p^0}.
\end{equation}
By \eqref{eq:Xcompare}, we further obtain
\begin{equation}
 N^{-1}\kappa\|u\|_{X_p^a}\leq \|u\|_{\widetilde X_\theta}
 \leq N\kappa^{-1}\|u\|_{X_p^a},                                 \label{eq:Xinterpolate}
\end{equation}
where $N$ is independent of $\kappa$.  
In particular, $p=2$ corresponds to $\theta_*=1/2$, and \eqref{eq:L2-lower} and \eqref{eq:Xinterpolate} give
\begin{equation*}
\|u\|_{\widetilde X_{1/2}}\le  N\kappa^{-1}\|T_a u\|_{\widetilde Y_{1/2}}.          
\end{equation*}
On the endpoint spaces, \eqref{eq:T-bound} gives an operator norm at most $\sqrt2$.  The quantitative Shneiberg theorem \cite[Theorem A.1]{ABESShneiberg} therefore implies that $T_a:\widetilde X_\theta\to\widetilde Y_\theta$ is invertible and
\begin{equation}\label{eqn-260920-1129}
    \|T_a^{-1}\|_{\widetilde Y_\theta\to \widetilde X_\theta}\le N\kappa^{-1}
    \quad \text{whenever}\,\, |\theta-1/2|<c\kappa
\end{equation}
for a small $c>0$ independent of $\kappa$.
Since $1/p=2/3-\theta/3$, after decreasing $c$ this is true for $p$ satisfying \eqref{eq:positive-range}.  Finally, \eqref{eqn-260920-1129}, \eqref{eq:Yinterpolate}, and \eqref{eq:Xinterpolate} give \eqref{eq:positive-est}.

Next, we consider divergence form equations. First, by considering $v = u_x$, we obtain the existence of solutions that satisfy the desired estimate. The uniqueness follows by duality. Indeed, suppose $v \in \cH^1_p(\cQ)$ is a solution of \eqref{eqn-260916-1128} with $f= 0$. For every $g \in C^\infty_c(\cQ)$, decreasing $c$ if needed, we can find $w \in \cH^1_{p'}(\cQ)$, satisfying
\begin{equation*}
    w_t - (a(1-t,\cdot) w_x)_x = g_x\quad \text{in}\,\,\cQ,
    \quad
    w(0,\cdot) = 0.
\end{equation*}
Testing the equation of $v$ by $\tilde w (t,\cdot):= w(1-t,\cdot)$, a standard approximation argument gives
\begin{equation*}
    \int_{\cQ} v_x(t,x) g(1-t,x)\,dxdt = - \int_{\cQ} v_t \tilde w + a v_x \tilde w_x \, dxdt = 0.
\end{equation*}
Since $g \in C^\infty_c(\cQ)$ is arbitrary, we conclude that $v_x\equiv 0$, and thus $v_t\equiv 0$ and $v\equiv 0$. 

By using the odd/even extensions, we obtain the corresponding results in the half space. Then the solvability for initial-boundary value problems follow from a standard partition of unity argument. See, for instance, \cite[\S 9]{Krylov2008}, We omit the details.
\end{proof}

\section{Construction and critical integrability}             \label{sec:construction}
\subsection{The local replacement}                              

We use the coordinates
\[
 (s,h)=(u_t,u_{xx})
\]
in $\bR^2$.  Define the admissible cone as
\begin{equation}
 \cC_\kappa=\{(s,h):h>0,\ \kappa h<s<\kappa^{-1}h\}.                  \label{eq:cone}
\end{equation}
At level $n$ we take four points in the $(s,h)$ space:
\begin{equation}
 \begin{split}
 A_n&=(n,n),\quad
 B_n=\left(n,1+\kappa(n-1)\right),\\
 D_n&=(n,n+1),\quad
 C_n=\left(1+\kappa n,n+1\right).
 \end{split}                                            \label{eq:points}
\end{equation}
Then, we have splittings
\[
A_n=\theta_n B_n+(1-\theta_n)D_n,
\qquad
\theta_n=\frac{1}{(1-\kappa)(n-1)+1},
\]
and
\[
D_n=\eta_n C_n+(1-\eta_n)A_{n+1},
\qquad
\eta_n=\frac{1}{(1-\kappa)n}.
\]

Choose \(n_0>1/(1-\kappa)\). Then, for every \(n\geq n_0\), all these four points lie strictly inside \(\cC_\kappa\), and
\(\theta_n,\eta_n\in(0,1)\).

In the following lemma, we give a local modification of a function $u$ on a rectangle where $u_{xx}=u_t=n$, i.e., $(u_t, u_{xx}) = A_n$.
The resulting function $\widetilde u$  is close to \(u\) in \(C^0\) and \(C_x^1\), while \((\widetilde u_t,\widetilde u_{xx})\) takes the values \(B_n, C_n\), and \(A_{n+1}\) on regions whose relative measures are close to
\begin{equation}
\theta_n,\qquad
\gamma_n=(1-\theta_n)\eta_n\quad
\text{and}\quad
\rho_n=(1-\theta_n)(1-\eta_n),
        \label{eq:fractions}
\end{equation}
respectively.
Note that as $n \rightarrow \infty$, 
\begin{equation}
                \label{eq3.22}
\theta_n, \gamma_n = \frac{1}{1-\kappa}n^{-1} + O(n^{-2})
\quad
\text{and}
\quad
\rho_n = 1 - \frac{2}{1-\kappa} n^{-1} + O(n^{-2}).
\end{equation}

\begin{figure}[h]
\centering
\begin{tikzpicture}[scale=.72,>=Latex]
\draw[->] (0,0) -- (7.2,0) node[right] {$s=u_t$};
\draw[->] (0,0) -- (0,7.2) node[above] {$h=u_{xx}$};
\draw[dashed] (0,0) -- (7,7) node[above right] {$h=s$};
\draw (0,0) -- (7,3.5) node[right] {$h=\kappa s$};
\draw (0,0) -- (3.5,7) node[above] {$h=s/\kappa$};
\coordinate (A) at (4,4);
\coordinate (B) at (4,2.5);
\coordinate (D) at (4,5);
\coordinate (C) at (3,5);
\coordinate (Ap) at (5,5);
\draw[thick] (B)--(D);
\draw[thick] (C)--(Ap);
\fill (A) circle (2.2pt) node[right] {$A_n$};
\fill (B) circle (2.2pt) node[right] {$B_n$};
\fill (D) circle (2.2pt) node[right] {$D_n$};
\fill (C) circle (2.2pt) node[above left] {$C_n$};
\fill (Ap) circle (2.2pt) node[above right] {$A_{n+1}$};
\end{tikzpicture}
\caption{The two convex splittings of points in the $(u_t,u_{xx})$-plane.  }
\label{fig:splitting}
\end{figure}
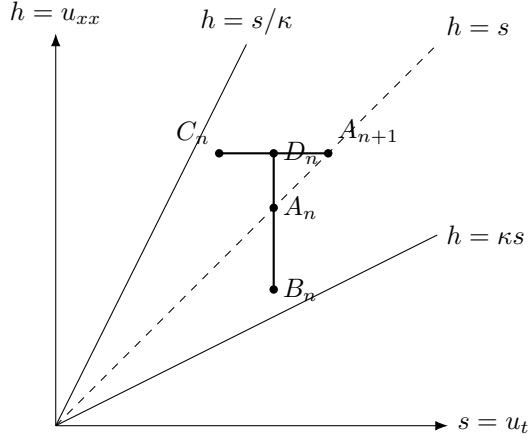

\begin{lemma}                         \label{lem:local}
Let $R$ be a bounded open rectangle, $n\geq n_0$, and let $u$ be smooth in $\overline R$ with
\begin{equation}
 u_t=u_{xx}=n\quad\text{in }R.                         \label{eq:local-base}
\end{equation}
Given any $\varepsilon,\delta>0$, there is a smooth function $\widetilde u$ such that
\begin{equation}
 \widetilde u-u\in C_0^\infty(R),\qquad
 \|\widetilde u-u\|_{L_\infty}
 +\|\widetilde u_x-u_x\|_{L_\infty}<\delta,           \label{eq:local-C1}
\end{equation}
and
\begin{equation}
 (\widetilde u_t,\widetilde u_{xx})\in\cC_\kappa,
 \quad \widetilde u_{xx}>\frac34,        
 \quad \max\{|\widetilde u_t|,|\widetilde u_{xx}|\}\le n+2
 \label{eq:local-cone}
\end{equation}
in $R$.  Outside a subset of measure at most $\varepsilon |R|$, $(\widetilde u_t,\widetilde u_{xx})$ equals $B_n$, $C_n$, or $A_{n+1}$ on finite unions of rectangles whose relative measures differ from the three numbers in \eqref{eq:fractions} by at most $\varepsilon$.
\end{lemma}

\begin{proof}
The identities \eqref{eq:local-base} imply
\[
 u(t,x)=n\left(t+\frac{x^2}{2}\right)+c_1 x+c_2
\]
in $R$ for some constants $c_1,c_2$.  Their values play no role below.

We first change the second-order spatial derivative.  
Extend periodically the step function 
$$
f (x) =
\begin{cases}
-(1-\kappa)(n-1), & 0< x<\theta_n,\\
1, & \theta_n<x<1.
\end{cases}
$$
By modifying \(f\) only on transition intervals of total length at most \(\varepsilon/M\) in each period, with a sufficiently large constant \(M\) to be specified later, we obtain a smooth \(1\)-periodic function \(q\) with zero mean and $q(x)\in [-(1-\kappa)(n-1),1]$. 

Choose a periodic function $b$ with $b''=q$, a cutoff $\chi\in C_0^\infty(R)$ which takes value in $[0,1]$ and equals one on a smaller inner rectangle $\tilde R$ that is $\epsi/M$ distance close to $R$, and set
\begin{equation*}
 u^{(1)}=u+\ell^2\chi(t,x)b(x/\ell).              
\end{equation*}
As $\ell\downarrow0$,
\begin{equation}
 u^{(1)}_t=n+O(\ell^2\epsi^{-1}M),\quad
 u^{(1)}_{xx}=n+\chi q(x/\ell)+O(\ell^2\epsi^{-2}M^2 + \ell \epsi^{-1}M),                  \label{eq:spatialder}
\end{equation}
while $u^{(1)}-u =O(\ell^2)$ and $u^{(1)}_x-u_x =O(\ell^2\epsi^{-1}M + \ell)$.  
Moreover, the relative portions of $(u^{(1)}_t, u^{(1)}_{xx}) = B_n$ and $D_n$ in $R$ are $\theta_n + O(\ell + \epsi/M)$ and $1-\theta_n + O(\ell + \epsi/M)$, respectively. Modulo a small set of relative measure $O(\ell+\epsi/M)$, the two level sets are given by finite disjoint unions of sub-rectangles.

On every sub-rectangle where the point is exactly $D_n$, we next change the time derivative. Similarly, extend periodically the step function 
$$
\tilde f (t) =
\begin{cases}
-((1-\kappa)n-1), & 0< t<\eta_n,\\
1, & \eta_n<t<1.
\end{cases}
$$
By modifying \(\tilde f\) only on transition intervals of total length at most \(\varepsilon/M\) in each period, we obtain a smooth \(1\)-periodic function \(\tilde q\) with zero mean. 

Choose a periodic function $\tilde b$ with $\tilde b' = \tilde q$. Now, on each sub-rectangle $R_i$, where $(u^{(1)}_t, u^{(1)}_{xx}) =D_n$, let $\chi_i \in C^\infty_0(R_i)$ be a cutoff function that equals one on a slightly smaller rectangle as before. In $R_i$, set
\begin{equation*}
\tilde u = u^{(1)} +\sigma \chi_i (t,x) \tilde b(t/\sigma).
\end{equation*}
As $\sigma \downarrow 0$,
\begin{equation}
 \widetilde u_t=n+ \chi_i \tilde q(t/\sigma)+O(\sigma \epsi^{-1}M),\qquad
 \widetilde u_{xx}=n+1+O(\sigma \epsi^{-2}M^2),                          \label{eq:temporalder}
\end{equation}
while $\tilde u - u^{(1)} = O(\sigma)$ and $\tilde u_x - u^{(1)}_x = O(\sigma \epsi^{-1}M)$. Moreover, the relative portions of $(\widetilde u_t, \widetilde u_{xx}) = C_n$ and $A_{n+1}$ in $R_i$ (and therefore, in the whole $\{(u^{(1)}_t, u^{(1)}_{xx}) = D_n\}$) are $\eta_n + O(\sigma + \epsi/M)$ and $1-\eta_n + O(\sigma + \epsi/M)$, respectively.

All the points in \eqref{eq:points} and line segments remain a positive distance from the boundary of $\cC_\kappa$ defined in \eqref{eq:cone} when $n\geq n_0$. Therefore, \eqref{eq:local-cone} holds if the errors in \eqref{eq:spatialder} and \eqref{eq:temporalder} are small enough. For this and also \eqref{eq:local-C1}, it suffices to choose $M$ in cutoffs large first and then the periods $\ell, \sigma$ sufficiently small.
\end{proof}

\subsection{Proof of Theorem~\ref{thm:main} (i)}

Let $n_0 > 1/(1-\kappa)$. We start from
\[
 u_{n_0}(t,x):=n_0\left(t+\frac{x^2}{2}\right),\qquad \cE_{n_0}:=D.
\]

Inductively, on each step we work on $\cE_n$, a finite union of rectangles on which
\begin{equation*}
 u_{n,t}=u_{n,xx}=n.
\end{equation*}
Apply Lemma \ref{lem:local} on every component of $\cE_n$.  The union of the rectangles on which the point is $A_{n+1}$ is denoted by $\cE_{n+1}$. From the construction, it is clear that $\overline \cE_{n+1}\subset \cE_{n}$.  The union of the $C_n$-rectangles is denoted by $\cG_n$. After this step, the value of the solution outside $\cE_{n+1}$ is never modified again.

Choose the errors in Lemma \ref{lem:local} small so that
\begin{equation}
 \|u_{n+1}-u_n\|_{L_\infty}
 +\|u_{n+1,x}-u_{n,x}\|_{L_\infty}\leq 2^{-n}.        \label{eq:summable}
\end{equation}
We may also arrange
\begin{equation}
 |\cE_{n+1}|=\rho_n(1+e_n)|\cE_n|,
 \qquad \sum_n|e_n|<\infty,                            \label{eq:measure-rec}
\end{equation}
and
\begin{equation}
 |\cG_n|\geq c\gamma_n|\cE_n|.                             \label{eq:G-lower}
\end{equation}
From \eqref{eq3.22},
\begin{equation}
 \rho_n
 =1-\frac{p}{n}+O(n^{-2}),
 \qquad p=\frac{2}{1-\kappa}.                              \label{eq:rho-exp}
\end{equation}
Using \eqref{eq:rho-exp} and taking logarithms in \eqref{eq:measure-rec},
\begin{equation}
 |\cE_n|\simeq n^{-p}.                                   \label{eq:E-size}
\end{equation}
Moreover, $\gamma_n\simeq n^{-1}$, and hence by \eqref{eq:G-lower},
\begin{equation}
 |\cG_n|\ge c n^{-p-1}.                                 \label{eq:G-size}
\end{equation}

We first prove the convergence of $u_n$.
Let $m>n$ and put
\[
 w=u_m-u_n.
\]
By construction,
\[
 w\in C_0^\infty(\cE_n),\qquad
 u_{n,t}=u_{n,xx}=n\quad\text{on }\cE_n.
\]
For any compactly supported smooth function,
\begin{equation*}
 \int w_t w_{xx}
 =-\int w_{tx}w_x
 =-\frac12\int\partial_t(w_x^2)=0.            
\end{equation*}
Since also $\int w_t=\int w_{xx}=0$, expanding the product gives
\begin{equation}
 \int_{\cE_n}u_{m,t}u_{m,xx}
 = \int_{\cE_n}(w_{t}+n)(w_{xx}+n)
 =n^2|\cE_n|.                  \label{eq:productid}
\end{equation}
By the local construction, we have
\begin{equation}\label{eqn-260912-1109}
u_{m,xx}\ge 3/4,\quad
\kappa \leq a_m \leq \kappa^{-1}\,\, \text{on}\,\,Q,
\quad
\text{where}\,\,a_m:= u_{m,t}/u_{m,xx}.
\end{equation}
 Hence \eqref{eq:productid} implies
\begin{equation}
 \int_{\cE_n}(|u_{m,t}|^2+|u_{m,xx}|^2)
 \leq N_\kappa n^2|\cE_n|.                                   \label{eq:tailL2}
\end{equation}
Since $u_m=u_n$ outside $\cE_n$, \eqref{eq:E-size} and \eqref{eq:tailL2} yield
\begin{equation*}
 \|u_{m,t}-u_{n,t}\|_{L_2}^2
 +\|u_{m,xx}-u_{n,xx}\|_{L_2}^2
 \leq N n^{2-p}\to0.             
\end{equation*}
Together with \eqref{eq:summable}, this gives
\begin{equation}
 u_n\to u\quad\text{strongly in }W^{1,2}_2(Q),
 \qquad
 u_n\to u,\quad u_{n,x}\to u_x\quad\text{uniformly}.\label{eq:strongconv}
\end{equation}

\medskip

Next, we show that the limit $u$ satisfies a desired equation. Let
\begin{equation*}
 \Sigma=\bigcap_{n\geq n_0}\overline{\cE_n}.           
\end{equation*}
By \eqref{eq:E-size}, $\Sigma$ is compact and $|\Sigma|=0$.  The rectangles may be chosen with compact containment at every stage, so $\Sigma\Subset D$.  Every point of $Q\setminus\Sigma$ has a neighborhood on which the sequence $u_n$ is eventually constant.  Thus $u$ is smooth there.  Define
\begin{equation}
 a(t,x)=
 \begin{cases}
 u_t/u_{xx},&(t,x)\in Q\setminus\Sigma,\\
 1,&(t,x)\in\Sigma.
 \end{cases}                                           \label{eq:limit-a}
\end{equation}
In view of \eqref{eq:strongconv}, the inequalities in \eqref{eqn-260912-1109} pass to the limit almost everywhere, so
\begin{equation} \label{eqn-260914-0706}
u_{xx}\ge 3/4,\quad
\kappa \leq a \leq \kappa^{-1}
\quad
\text{and}
\quad u_t - a u_{xx} = 0,\,\,\text{a.e.  in}\,\,Q.
\end{equation}

Finally, we show the desired growth of level sets. 
Outside $\cE_j$ the construction has stopped by level $j$, and therefore
\[
 |u_t|+|u_{xx}|\leq N j
 \quad\text{a.e. on }Q\setminus \cE_j.
\]
It follows from \eqref{eqn-260914-0706} and \eqref{eq:E-size} that
\begin{equation*}
 |\{|u_{xx}|>N j\}|
 \leq N j^{-p}.                                 
\end{equation*}
Conversely, for $n\geq j$, $u_{xx}=n+1$ on $\cG_n$, and hence by \eqref{eq:G-size},
\begin{equation*}
 |\{u_{xx}>j\}|
 \geq\sum_{n\geq j}|\cG_n|
 \geq cj^{-p}.                            
\end{equation*}

This finishes the construction in Theorem \ref{thm:main} (i).

\subsection{Proof of Theorem \ref{thm:main} (ii)}              \label{sec:estfailure}

Keep the coefficient \eqref{eq:limit-a} fixed.  Outside $\cE_j$, the function $u_j$ has already stabilized to $u$.  On $\cE_j$,
\[
 u_{j,t}=u_{j,xx}=j.
\]
Thus
\begin{equation*}
 h_j:=u_{j,t}-a u_{j,xx}
 =j(1-a)\mathbf 1_{\cE_j}.                   
\end{equation*}
By \eqref{eq:E-size},
\begin{equation}
 \|h_j\|_{L_p(Q)}^p\leq N j^p|\cE_j|\leq N.             \label{eq:hjbound}
\end{equation}
On the other hand,
\begin{equation}
 \int_D|u_{j,xx}|^p
 \geq c\sum_{n=n_0}^{j-1}(n+1)^p|\cG_n|
 \geq c\log j-N.                                      \label{eq:ujlog}
\end{equation}
Choose $\chi\in C_0^\infty(Q)$ equal to one in a neighborhood of $D$ and set
$\phi_j=\chi u_j$.
Outside $D$, all $u_j$ equal the same quadratic polynomial and $a=1$.  Hence the derivatives of $\chi$ create one fixed smooth error, independent of $j$.  If
\begin{equation*}
 f_j:=\phi_{j,t}-a\phi_{j,xx},                      
\end{equation*}
then \eqref{eq:hjbound} and \eqref{eq:ujlog} prove \eqref{eq:test-upper}--\eqref{eq:test-lower}.

\section{The lower exponent by duality}                  \label{sec:duality}

In this section, we prove Theorem \ref{cor:dual}.

Let
\[
 p_+=\frac{2}{1-\kappa},\quad p_-=\frac{2}{1+\kappa}.
\]
Translate the coefficient constructed above into $\cQ=(0,1)\times\bR$ and extend it by $1$ outside the translated copy of $Q$.  We keep the notation $a$ for the extended coefficient.  By Theorem \ref{thm:main}, the estimate
\begin{equation}
 \|u_{xx}\|_{L_{p_+}(\cQ)}\leq N\|u_t-a u_{xx}\|_{L_{p_+}(\cQ)}       \label{eq:q-failure}
\end{equation}
fails for compactly supported smooth functions.

Set
\[
 \widetilde a(t,x)=a(1-t,x).
\]
We prove that the corresponding estimate at exponent $p_-$ also fails.  At exponent $2$, for $f\in L_2(\cQ)$ let $u$ be the unique energy solution of
\begin{equation}
 u_t-\widetilde a u_{xx}=f,\qquad u(0,\cdot)=0,                 \label{eq:forward-tilde}
\end{equation}
and define
\[
 S_{\widetilde a}f=u_{xx}.
\]
The $W^{1,2}_2$ estimate shows that $S_{\widetilde a}$ is bounded on $L_2(\cQ)$.

Suppose, to the contrary, that the $W^{1,2}_{p_-}$ a priori estimate \eqref{eq:lower-forward} holds for the $W^{1,2}_2$ solutions of \eqref{eq:forward-tilde}.  Then, for $f\in L_2(\cQ)\cap L_{p_-}(\cQ)$,
\begin{equation}
 \|S_{\widetilde a}f\|_{L_{p_-}(\cQ)}\leq N\|f\|_{L_{p_-}(\cQ)}.       \label{eq:r-estimate}
\end{equation}
By density, $S_{\widetilde a}$ extends to a bounded operator on $L_{p_-}(\cQ)$.  Hence its adjoint operator $S_{\widetilde a}^*$ is bounded on $L_{p_+}(\cQ)$.

For $g\in L_2(\cQ)$, let $w$ be the $W^{1,2}_2$ solution of
\begin{equation}
 -w_t-\widetilde a w_{xx}=g,\qquad w(1,\cdot)=0.                \label{eq:backward-tilde}
\end{equation}
By a standard approximation argument, integration by parts gives
\begin{equation*}
 \int_{\cQ}u_{xx}g=\int_{\cQ}f w_{xx}.                   
\end{equation*}
Here it is crucial that the space dimension is $1$.
Consequently,
\[
 S_{\widetilde a}^*g=w_{xx}.
\]
After the change of variables $t\mapsto1-t$, equation \eqref{eq:backward-tilde} becomes the forward equation with coefficient $a$.  Thus the boundedness of $S_{\widetilde a}^*$ on $L_{p_+}$ gives, after this change of variables, the estimate \eqref{eq:q-failure}, a contradiction.  Thus \eqref{eq:r-estimate}, and hence the $W^{1,2}_{p_-}$ estimate, cannot hold.  

The statement for divergence form equations follows in exactly the same way by using Corollary \ref{cor:div}.
Theorem \ref{cor:dual} is proved.

\section{Alexandrov maximum principle}\label{sec-alex}

\begin{proof}[Proof of Theorem \ref{thm-alex}]
We start with the positive result in (i). This follows from the $W^{1,2}_p(Q_1)$ solvability of initial-boundary value problems in Theorem \ref{thm:positive}, approximation,
the standard Alexandrov maximum principle for $W^{1,2}_2$ strong solutions,  and the parabolic Sobolev embedding $W^{1,2}_p \hookrightarrow C^0$ for $p>3/2$. The procedure is standard and we omit the details.

Next, we prove the negative result in (ii).
For $u$, $a$, and $D$ defined in the Theorem \ref{thm:main}, let $\widetilde u(t,\cdot) = u(-1-t, \cdot)$, $\widetilde a(t,\cdot) = a(-1-t,\cdot)$ and $\widetilde D:= \{(t,x): (-1-t,x) \in D\}$. Then we have
\begin{equation*}
    \widetilde u_t + \widetilde a \widetilde u_{xx} = 0\quad \text{in}\,\,Q_1,
    \quad
    \widetilde u_{xx} \notin L_{p_+}(\widetilde D),
\end{equation*}
and
\begin{equation}\label{eqn-260918-1145-1}
    \widetilde u = n_0 (-1-t + \frac{x^2}{2}) + cx + d, \quad \widetilde a = - \widetilde u_t/\widetilde u_{xx} = 1 \quad \text{in}\,\,Q_1\setminus \widetilde D.
\end{equation}

Let $\eta \in C^\infty_c(Q_1)$ be a usual cutoff function with $\eta = 1$ on $\widetilde D$. Thus
\begin{equation}\label{eqn-260918-1145-2}
   \operatorname{supp} \eta_t,\, \operatorname{supp} \eta_x \subset Q_1\setminus \widetilde D.
\end{equation}
Note that $\eta \widetilde u$ satisfies
\begin{equation*}
    (\eta \widetilde u)_t + \widetilde a (\eta \widetilde u)_{xx} = \eta_t \widetilde u + \widetilde a (2\eta_x \widetilde u_x + \eta_{xx} \widetilde u) =: g \quad \text{in}\,\, Q_1.
\end{equation*}
By \eqref{eqn-260918-1145-1} and \eqref{eqn-260918-1145-2}, $g \in C^\infty_c(Q_1)$.

Now, suppose the contrary that \eqref{eqn-260918-1142} holds for some $p \in (1,p_-]$ and every $f \in L_2(Q_1)$. 
As in the proof of Theorem \ref{cor:dual}, a duality argument applied to $v$ and $\eta \widetilde u$ gives
\begin{equation*}
    \int_{Q_1} (\eta \widetilde u)_{xx} f \,dx dt= -\int_{Q_1} v_{xx} g\,dx dt = -\int_{Q_1} v g_{xx}\,dx dt.
\end{equation*}

By \eqref{eqn-260918-1142}, 
\begin{equation*}
    \left|\int_{Q_1} (\eta \widetilde u)_{xx} f \,dx dt\right| = \left|\int_{Q_1} v g_{xx}\,dx dt\right| \leq \|v\|_{L_1(Q_1)} \|g_{xx}\|_{L_\infty(Q_1)}
    \leq
    N_g \|f\|_{L_{p_-}(Q_1)}.
\end{equation*}
Since $f \in L_2(Q_1)$ is arbitrary, this implies $(\eta \widetilde u)_{xx} \in L_{p_+}(Q_1)$, and therefore, $\widetilde u_{xx} \in L_{p_+}(D)$. 
This contradicts $\widetilde u_{xx} \notin L_{p_+}(D)$, and therefore, (ii) is proved.
\end{proof}

\section*{acknowledgement}

The first author would like to thank Nicolai V. Krylov for bringing this problem to his attention a few years ago and helpful comments on an earlier version of the manuscript.
The authors acknowledge the assistance of generative AI tools in the development of the counterexample in Theorem \ref{thm:main}, particularly in drawing our attention to its connection with \cite{AFS2008}. The authors take full responsibility for the content of the manuscript.

\def\cprime{$'$}

\end{document}